\documentclass{article}

\usepackage{amsmath}
\usepackage{amsfonts}
\usepackage{amsthm}

\DeclareMathOperator{\Cl}{Cl}
\DeclareMathOperator{\Span}{Span}

\DeclareMathOperator{\assoc}{\mathcal{A}}
\DeclareMathOperator{\conjassoc}{\mathcal{A}^{--}}
\DeclareMathOperator{\conjassocL}{\mathcal{A}_{\mathit{L}}^{--}}
\DeclareMathOperator{\conjassocR}{\mathcal{A}_{\mathit{R}}^{--}}
\DeclareMathOperator{\cycl}{\mathcal{C}}
\DeclareMathOperator{\cyclplus}{\mathcal{C}^{+}}
\DeclareMathOperator{\cyclminus}{\mathcal{C}^{--}}

\newtheorem{theorem}{Theorem}[section]

\newtheorem{corollary}[theorem]{Corollary}
\newtheorem{definition}[theorem]{Definition}

\newtheorem{lemma}[theorem]{Lemma}

\title{On cyclic left modules and arbitrary bimodules over the octonions}

\author{%
M\'at\'e Lehel Juh\'asz\footnote{Partial support from Project 2018-1.2.1-NKP-00008: Exploring the Mathematical Foundations of Artificial Intelligence and the "Big Data—--Momentum" grant of the Hungarian Academy of Sciences. }
}

\begin{document}

\maketitle

\begin{abstract}
Qinghai Huo, Yong Li, Guangbin Ren described the structure of left $\mathbb{O}$-modules in great detail in \cite{huo2019}.
However, they left open a question on cyclic left $\mathbb{O}$-modules.
This note intends to close this gap and provide a complete description of bimodules.
\end{abstract}

\section{Introduction}

In \cite{huo2019}, Qinghai Huo, Yong Li and Guangbin Ren set out to classify a generalization of vector spaces over the non-associative algebra known as the octonions (denoted by $\mathbb{O}$) which they call left $\mathbb{O}$-modules.
They construct a fairly complete theory but end on a conjecture about the nature of submodules generated by a single element.
This short note provides an answer to that question, completing their theory.
Furthermore, in the rest of the note it is shown that by imposing a natural bimodule structure, left modules that correspond to vector spaces have a unique bimodule structure (in stark contrast to the quaternionic case), while other left modules have none.

The octonions are part of a widely known series of real division algebras that can be created iteratively via the Cayley-Dickson construction, and includes the real numbers $\mathbb{R}$, complex numbers $\mathbb{C}$ and quaternions $\mathbb{H}$.
In fact, the octonions with these $3$ provide the only possible examples of division algebras over the real numbers.
Unlike the first three however, octonions lack the associative property and instead satisfy a weaker property called the \textbf{alternative identities}:
\[
(xx)y=x(xy)
\qquad
(xy)x=x(yx)
\qquad
(yx)x=y(xx).
\]
This means that the structure of modules becomes more complicated, since consecutive multiplication of a vector by two scalars cannot be simplified to a single multiplication: $\lambda(\mu\mathbf{v})\ne(\lambda\mu)\mathbf{v}$.
If we permitted this, only the trivial module would be allowed.
Instead, a similar law can be proposed using alternativity: $\lambda(\lambda\mathbf{v})=(\lambda\lambda)\mathbf{v}$.
The usual method for classifying modules by choosing a set of vectors as basis will not work with this rule.
Nonetheless with this condition \cite{huo2019} builds up a complete theory of such modules.

One key fact in \cite{huo2019} is that every such module has a natural Clifford module structure.
This permits one to completely describe an $\mathbb{O}$-module using the dimension of two real subspaces of the module, the associative and conjugate associative parts, which vaguely correspond to vectors containing only real entries.
These two dimensions completely characterize such a module.
For example, there are only two types of simple $\mathbb{O}$-modules: one that is isomorphic to $\mathbb{O}$ as a module over itself, and one that is isomorphic to its conjugate, $\overline{\mathbb{O}}$, where the multiplication rule $\star$ is given by $p\star\mathbf{m}=\overline{p}\mathbf{m}$.
Then ${}_{\mathbb{O}}\mathbb{O}$ has a trivial conjugate associative part, while ${}_{\mathbb{O}}\overline{\mathbb{O}}$ has a trivial associative part.

Section 2 is mostly a reiteration of those definitions and theorems in \cite{huo2019} without proofs that are needed for the main theorem.
A few facts are included to give a more explicit description of left $\mathbb{O}$-modules.
These will be pointed out specifically.
Section 3 gives a characterization of cyclic left $\mathbb{O}$-modules.
Section 4 presents a definition of bimodules and shows that only those left $\mathbb{O}$-modules that have no conjugate associative part admit a bimodule structure, and that that structure is unique and isomorphic to $\bigoplus_{i\in S}\mathbb{O}$ for some set $S$.

\section{Review of previous results}

Much of this section is a reiteration of the results in \cite{huo2019}, without proofs.
The reader is kindly asked to refer to the article for the proofs.
Statements not found there will be pointed out explicitly.

The \textbf{octonions} $\mathbb{O}$ are an $8$ dimensional real algebra, generated by $1,e_1,\dots,e_7$ such that $e_i^2=-1$ and $e_ie_j=-e_je_i=e_k$ whenever $k-i\equiv 3(j-i)~(\mathrm{mod}~7)$.
They satisfy the alternative property, which can be reformulated using the \textbf{associator}: $[a,b,c]=(ab)c-a(bc)$ as
\[
[a,b,c]=[b,c,a]=[c,a,b]=-[b,a,c].
\]
The \textbf{octonionic conjugation} is a real linear transformation that extends the map that preserves $1$ but maps all $e_i$ to the corresponding $-e_i$.
It is denoted by $p\mapsto \overline{p}$.
An octonion is \textbf{purely imaginary} if $\overline{p}=-p$.
The map $p\mapsto p\cdot\overline{p}$ has real, non-negative values, and defines a positive definite quadratic form on $\mathbb{O}$ considered as an $8$-dimensional real vector space.

\begin{definition}
(\cite{huo2019}, 3.1)
A (left) \textbf{$\mathbb{O}$-module} is a real vector space $M$ equipped with an $\mathbb{R}$-bilinear map $\cdot\colon\mathbb{O}\times M\to M$ denotes as $(q,m)\mapsto qm$ with the following conditions.
Let
\[
[p,q,m]:=(pq)m-p(qm).
\]
Then
\begin{itemize}

\item $1m=m$;

\item $[p,q,m]=-[q,p,m]$, the \textbf{alternating rule}.

\end{itemize}
\end{definition}

Let us denote the operation of left multiplication by $q$ on a left $\mathbb{O}$-module by $L_q$.
Then by the alternating rule,
\[
L_{e_i}L_{e_j}+L_{e_j}L_{e_i}=-2\delta_{ij}I
\]
where $\delta_{ij}$ is the Kronecker symbol where $\delta_{ij}=1$ only if $i=j$, otherwise it is $0$, and $I$ is the identity operation.

Given a real vector space $V$ with a non-degenerate quadratic form $V$ over it, the \textbf{Clifford algebra} corresponding to $(V,Q)$ is the algebra generated by vectors in $V$ with the rule that $v^2=-Q(v)\cdot 1$.
A formal definition is given in \cite{atiyah1963}.

The algebra generated by $L_{e_1},\dots,L_{e_7}$ is a Clifford algebra over a $7$ dimensional Euclidean space, where the operators $L_{e_i}$ provide an orthogonal basis.
We will denote this algebra by $\Cl_7(\mathbb{R})$.

\begin{lemma}
(\cite{huo2019}, 4.1)
The category of left $\mathbb{O}$-modules is isomorphic to the category of left $\Cl_7(\mathbb{R})$-modules.
\end{lemma}

Let us denote by $M(n,\mathbb{F})$ the ring of $n\times n$ matrices over a field $\mathbb{F}$.

\begin{lemma}
(\cite{atiyah1963})
$\Cl_7(\mathbb{R})$ is isomorphic to the direct sum of two matrix algebras, $M(8,\mathbb{R})\oplus M(8,\mathbb{R})$.
\end{lemma}

Note that $\mathbb{O}$ is a simple left $\mathbb{O}$-module where scalar multiplication is given by the octonionic multiplication.
Let us define the left $\mathbb{O}$-module $\overline{\mathbb{O}}$, isomorphic as a real vector space to $\mathbb{O}$, endowed with the multiplication
\[
p\star m := \overline{p}\cdot m
\]
where $\cdot$ is the octonionic multiplication.

The following is a classically known result arising from the Morita equivalence between a ring and a matrix ring over it.
Consider that the ring $M(8,\mathbb{R})$ acts on $\mathbb{R}^8$ via matrix multiplication.
Then given a real vector space $U$, $M(8,\mathbb{R})$ acts naturally on $\mathbb{R}^8\otimes U\cong U^{\oplus 8}$.

\begin{lemma}
Every left $\Cl_7(\mathbb{R})$-module is isomorphic to $U^{\oplus 8}\oplus V^{\oplus 8}$ for some real vector spaces $U$ and $V$ where the two direct summands in $M(8,\mathbb{R})\oplus M(8,\mathbb{R})$ act on the two components.
\end{lemma}

\begin{lemma}
(\cite{huo2019})
The only simple $\Cl_7(\mathbb{R})$-modules are $(\mathbb{R}^8,0)$ and $(0,\mathbb{R}^8)$, which correspond to $\mathbb{O}$ and $\overline{\mathbb{O}}$.
\end{lemma}

\begin{corollary}
The map
\[
p\in\mathbb{O}\mapsto L_p|_{\mathbb{O}}\oplus L_p|_{\overline{\mathbb{O}}},
\]
gives an explicit representation of the elements of $\mathbb{O}$ in $\Cl_7{\mathbb{R}}\cong M(8,\mathbb{R})\oplus M(8,\mathbb{R})$, compatible with the left $\mathbb{O}$-module structures of $\mathbb{O}$ and $\overline{\mathbb{O}}$.
\end{corollary}

\begin{definition}
Given a left $\mathbb{O}$-module $M$, the set of \textbf{associative elements} is defined as
\[
\assoc(M):=\{m\in M\mid [p,q,m]=(pq)m-p(qm)=0, \forall p,q\in\mathbb{O}\}.
\]
The set of \textbf{conjugate associate elements} is defined as
\[
\conjassoc(M):=\{m\in M\mid (pq)m-q(pm)=0, \forall p,q\in\mathbb{O}\}.
\]
\end{definition}

The next theorem is a combination of several results in \cite{huo2019}, including 3.10, 3.13 and 4.5.

\begin{theorem}
\label{thm:decomp}
Any left $\mathbb{O}$-module $M$ decomposes as the direct sum
\[
M=\mathbb{O}\assoc(M)\oplus\mathbb{O}\conjassoc(M).
\]
Furthermore, any real basis $S^{+}$ of $\assoc(M)$ (respectively, $S^{-}$ for $\conjassoc(M)$) we have $\mathbb{O}\assoc(M)\cong\bigoplus_{i\in S^{+}} \mathbb{O}$ (respectively, $\mathbb{O}\conjassoc(M)\cong\bigoplus_{i\in S^{-}} \overline{\mathbb{O}}$), and any element $m\in M$ can be written uniquely as an $\mathbb{O}$-linear combination of the elements of $S:=S^{+}\cup S^{-}$.

The pair
\[
(\dim_{\mathbb{R}}\assoc(M),\dim_{\mathbb{R}}\conjassoc(M))
\]
uniquely determines the module up to isomorphism and is called the \textbf{type} of the module.
We shall denote $M^{+}=\mathbb{O}\assoc(M)$ and $M^{-}=\mathbb{O}\conjassoc(M)$.
\end{theorem}

While not included in the statement of 4.5 in \cite{huo2019}, their proof also shows that $S$ can be chosen as the union of an arbitrary basis of the real vector spaces $\assoc(M)$ and $\conjassoc(M)$, a fact that we will rely on later on.

An important and noteworthy corollary not stated explicitly in \cite{huo2019} is the following.

\begin{corollary}
\label{thm:assoc}
When writing a left $\Cl_7(\mathbb{R})$-module as $M=U^{\oplus 8}\oplus V^{\oplus 8}$, we may identify $\assoc(M)$ with $U$ (or more precisely, $(U\oplus 0^{\oplus 7})\oplus 0^{\oplus 8}$) and $\conjassoc(M)$ with $V$ (or more precisely, $0^{\oplus 8}\oplus (U\oplus 0^{\oplus 7})$).
Furthermore, we may identify $U^{\oplus 8}$ with $U\oplus \bigoplus_{i=1}^7 e_iU$ and $V^{\oplus 8}$ with $V\oplus \bigoplus_{i=1}^7 e_iV$.
\end{corollary}

\begin{definition}
(\cite{huo2019}, 4.9)
An element $m$ of a left $\mathbb{O}$-module $M$ is referred to as \textbf{cyclic} if the generated cyclic submodule $\langle m\rangle_{\mathbb{O}}$ is $\mathbb{O}m$.
Equivalently, it is cyclic if $\dim_{\mathbb{R}}\langle m\rangle_{\mathbb{O}}=8$, therefore isomorphic to either $\mathbb{O}$ or $\overline{\mathbb{O}}$.
We define
\[
\cyclplus(M):=\{m\in\cycl(M)\mid \langle m\rangle_{\mathbb{O}}\cong\mathbb{O}\}\cup\{0\}
\]
\[
\cyclminus(M):=\{m\in\cycl(M)\mid \langle m\rangle_{\mathbb{O}}\cong\overline{\mathbb{O}}\}\cup\{0\}
\]
\[
\cycl(M):=\cyclplus(M)\cup\cyclminus(M)
\]
\end{definition}

\begin{theorem}
(\cite{huo2019}, 4.16)
\label{thm:cyclid}
We have the following identifications:
\[
\cyclplus(M)=\bigcup_{p\in\mathbb{O}} p\cdot\assoc(M)
\qquad
\cyclminus(M)=\bigcup_{p\in\mathbb{O}} p\cdot\conjassoc(M).
\]
\end{theorem}

\begin{theorem}
(\cite{huo2019}, 4.17)
For a left $\mathbb{O}$-module $M$
\[
M=\Span_{\mathbb{R}}\cycl(M)=\Span_{\mathbb{R}}\cyclplus(M)\oplus \Span_{\mathbb{R}}\cyclminus(M).
\]
\end{theorem}

\begin{lemma}
(\cite{huo2019}, 4.6)
\label{thm:findim}
Given an arbitrary left $\mathbb{O}$-module $M$ and an element $m\in M$, the real dimension of the generated submodule $\langle m\rangle$ is finite and at most $128$.
\end{lemma}

\begin{definition}
Given an element $m$ in a left $\mathbb{O}$-module $M$, let us denote its decomposition in $M=M^{+}\oplus M^{-}$ by $m=m^{+}\oplus m^{-}$, and write them as sums
\[
m=\left(\sum_{i=1}^n m_i^{+}\right)\oplus\left(\sum_{i=1}^{n'} m_i^{-}\right)
\]
where $m_i^{+}\in\cyclplus(M)$ and $m_i^{-}\in\cyclminus(M)$.
There is a smallest number $\ell_m^{+}$ for $n$ and $\ell_m^{-}$ for $n'$ where this decomposition is possible, and the pair $(\ell_m^{+},\ell_m^{-})$ shall be referred to as the \textbf{length of the element $m$}.
\end{definition}

We finally arrive at the final conjecture of Huo, Li and Ren.

\begin{theorem}
(\cite{huo2019}, 4.18)
Let $M$ be a left $\mathbb{O}$-module and $m\in M$ be an element of length $(\ell_m^{+},\ell_m^{-})$.
Then
\[
\langle m\rangle_{\mathbb{O}}\cong \mathbb{O}^{\ell_m^{+}}\oplus \overline{\mathbb{O}}^{\ell_m^{-}}
\]
\end{theorem}

In the next section, we provide a proof for this statement.

\section{Cyclic left $\mathbb{O}$-modules}

First notice that given a left $\mathbb{O}$-module, the set of its left submodules is in a one-to-one correspondence with pairs of real subspaces of $\assoc(M)$ and $\conjassoc(M)$.

\begin{theorem}
\label{thm:submod}
Let $M$ be a left $\mathbb{O}$-module and $U\le\assoc(M)$ and $V\le\conjassoc(M)$ real subspaces.
Define $M':=\mathbb{O}U\oplus\mathbb{O}V$ as the generated submodule.
Then $\assoc(M')=U$ and $\conjassoc(M')=V$.
\end{theorem}

\begin{proof}
It can be checked by a quick verification that $M'$ is in fact a left $\mathbb{O}$-module.

By the nature of their definitions, $\assoc(M')\le\assoc(M)$ and $\conjassoc(M')\le\conjassoc(M)$ as real vector spaces.
Hence $U\subseteq\assoc(M')$ and $V\subseteq\conjassoc(M')$.

For any $m\in M'$ we have $m=\sum_i p_iu_i\oplus \sum_i q_iv_i$ for some linearly independent elements $u_i\in U$ and $v_i\in V$.
Assume that $m\in\assoc(M')\setminus U\subseteq\assoc(M)$, which means that the set $\{m,u_i\}$ is $\mathbb{R}$-linearly independent in $\assoc(M)$.
The set $\{m,u_i,v_i\}$ can be extended into a basis of $M$.
Then $0=\sum_i p_iu_i+\sum_i q_iv_i-m$ gives a non-trivial decomposition of $0$ in $M'$, which contradicts \ref{thm:decomp} for $M$.
\end{proof}

\begin{lemma}
\label{thm:maxtype}
Let $M$ be of type $(n,k)$ with $n$, $k$ finite, and an element $m\in M$ of length $(n',k')$.
Then $n'\le n$ and $k'\le k$.
\end{lemma}

\begin{proof}
By \ref{thm:decomp}, we may choose bases for the real vector spaces $e_1,\dots,e_n\in\assoc(M)$ and $f_1,\dots,f_k\in\conjassoc(M)$, and any $m\in M$ decomposes uniquely as
\[
m=\sum_ip_i e_i\oplus\sum_iq_i f_i
\]
By \ref{thm:cyclid}, all the terms $p_ie_i$ appear in $\cyclplus(M)$, and all the terms $q_if_i$ appear in $\cyclminus(M)$, and so this decomposition of $m$ gives an upper limit for its length as $(n,k)$.
\end{proof}

With these two lemmas, we can prove the conjecture in \cite{huo2019}.

\begin{theorem}
For an arbitrary module $M$ and element $m$ of length $(n,k)$, the generated submodule is of type $(n,k)$.
\end{theorem}

\begin{proof}
Since $m$ is of type $(n,k)$, it has a decomposition
\[
m=\sum_{i=1}^n m_i^{+}\oplus\sum_{j=1}^k m_j^{-}
\]
where $m_i^{+}\in\cyclplus(M)$ and $m_j^{-}\in\cyclminus(M)$.
By \ref{thm:cyclid}, these can further be written as $m_i^{+}=p_ie_i$ and $m_j^{-}=q_jf_j$ for $p_i,q_j\in\mathbb{O}$, $e_i\in\assoc(M)$ and $f_i\in\assoc(M')$.

Let us consider the set $M':=\sum_{i=1}^n \mathbb{O}e_i\oplus\sum_{i=1}^n \mathbb{O}f_i$.
This is a submodule, since all $e_i\in\assoc(M)$ and $f_i\in\conjassoc(M)$.
Furthermore, by \ref{thm:submod}, $\assoc(M')=\Span_{\mathbb{R}}(e_1,\dots,e_n)$ and $\conjassoc(M')=\Span_{\mathbb{R}}(f_1,\dots,f_k)$.
Therefore the type of $M'$ is $(n',k')$ where $n':=\dim_{\mathbb{R}}\assoc(M')\le n$ and $k':=\dim_{\mathbb{R}}\conjassoc(M')\le k$.
However, since $m\in M'$ is of length $(n,k)$, by \ref{thm:maxtype}, these two must be equalities, and $n=n'$, $k=k'$.
Since the submodule $\langle m\rangle\subseteq M'$ has length at least $(n,k)$, \ref{thm:submod} shows that $\assoc(\langle m\rangle)=\assoc M$ and $\conjassoc(\langle m\rangle)=\conjassoc M$, hence $M'$ is the generated submodule.
\end{proof}

\begin{corollary}
The set of elements $m$ of $M$ of finite type $(n,k)$ that generate $M$ are precisely those elements whose length is $(n,k)$.
\end{corollary}

We can go further and characterize the length of elements by looking at what is essentially their octonionic coordinates when expressed as elements of $\mathbb{O}^n\oplus\overline{\mathbb{O}}^k$.

\begin{lemma}
\label{thm:typeisrank}
Given a left $\mathbb{O}$-module $M$, consider a basis of $M$, $\{u_1,\dots\}\subset\assoc(M)$, $\{v_1,\dots\}\subset\conjassoc(M)$.
Let $m\in M$ of length $(n',k')$ have a decomposition $\sum_i p_iu_i+\sum_j q_jv_j$.
Then $n'$ is the rank of the set of vectors $\{p_i\}$ in $\mathbb{O}$ as a real subspace, and $k'$ is that of $\{q_i\}$.
\end{lemma}

\begin{proof}
Because of \ref{thm:findim}, we may assume without loss of generality that $M$ is of finite type $(n,k)$.
Hence $M$ is isomorphic to $\mathbb{O}^n\oplus\overline{\mathbb{O}}^k$.
Let us recall \ref{thm:assoc} to consider the structure of $M$ as a left $\Cl_7(\mathbb{R})$-module where $\Cl_7\cong M(8,\mathbb{R})\oplus M(8,\mathbb{R})$.
Let us denote the left matrix algebra $M(8,\mathbb{R})$ by $M_L$, the right matrix algebra by $M_R$.
The right algebra $M_R$ acts trivially on $\mathbb{O}$, while $M_L$ acts trivially on $\overline{\mathbb{O}}$.
Since the only simple module over $M(8,\mathbb{R})$ is $\mathbb{R}^8$, we may identify the module $\mathbb{O}$ with $\mathbb{R}^8$ via the canonical real basis $1,e_1,\dots,e_7$ of $\mathbb{O}$.
Similarly for $\overline{\mathbb{O}}$.
Thus an element $m\in M$ may be identified as an $n+k$-tuple of real vectors in $\mathbb{R}^8$, identified by the real vectors $p_i$ and $q_i$ as expressed in the statement.

Now assume that the rank of the set of real vectors $p_i$ is $n'$ and that of $q_i$ is $k'$, and we will show that the submodule generated by $m$ is of type $(n',k')$.
By multiplying $m$ by $I\oplus 0\in M_L\oplus M_R$, we get the $n+k$-tuple consisting of $p_i$s and all $0$s.
Let us take a linearly independent subset $p_{i_1},\dots,p_{i_{n'}}$ of $\{p_1,\dots,p_{n'}\}$.
For each $s$, there is an element $A_s$ of $M_L$ that sends the vectors $p_{i_t}$ to $0$ for all $t\ne s$ except for $p_{i_s}$ which it sends to $1\in\mathbb{O}\cong\mathbb{R}^8$.
Furthermore $p_{i_s}\in\mathbb{O}$ itself corresponds to an element $L_{p_{i_s}}\in\Cl_7(\mathbb{R})$.
Then for any $p_i$, we have the identity $p_i=\sum_{s=1}^{n'} L_{p_{i_s}}\cdot (A_s\oplus 0)\cdot p_i$, as $\sum_{s=1}^{n'} L_{p_{i_s}}(A_s\oplus 0)$ is the identity map over $\Span_{\mathbb{R}}\{p_1,\dots,p_{n'}\}\oplus 0$.
Similarly we can choose a linearly independent set $q_{i_1},\dots,q_{i_{k'}}$ and matrices $B_s$ to get $q_i=\sum_{s=1}^{k'} L_{q_{i_s}}\cdot (0\oplus B_s)\cdot q_i$.

Put together we get $m=\sum_{s=1}^{n'} L_{p_{i_s}}\cdot((A_s\oplus 0)m)+\sum_{s=1}^{k'} L_{q_{i_s}}\cdot((0\oplus B_s)m)$, therefore $m$ is in the submodule generated by the elements $(A_s\oplus 0)m$ and $(0\oplus B_s)m$.
These in turn are in the submodule generated by $m$.
They are also linearly independent, since $(A_s\oplus 0)m$ has a zero coordinate for all $i_t$ where $t\ne s$ but non-zero in the coordinate $i_s$, and similarly for $(0\oplus B_s)m$.
Since the image of $A_s\oplus 0$ is in $\assoc(M)$ and that of $0\oplus B_s$ is in $\conjassoc(M)$, the module generated by $m$ is in fact of type $(n',k')$.
\end{proof}

\begin{corollary}
The set of cyclic left $\mathbb{O}$-modules are those that have types $(n,k)$ where $0\le n,k\le 8$.
\end{corollary}

\begin{proof}
Using \ref{thm:typeisrank}, we see that the rank of a set of octonions considered as real vectors is at most $8$, therefore the upper bound holds.
Take an arbitrary pair of integers $(n,k)$ satisfying the stated inequalities.
Let the real basis of $\mathbb{O}$ be denoted by $1,e_1,\dots,e_7$.
Within $\mathbb{O}^n\oplus\overline{\mathbb{O}}^k$, choose the element
\[
m=[1,e_1,\dots,e_{n-1}]\oplus[1,e_1,\dots,e_{k-1}]
\]
(writing $[1]$ in the case of $n=1$ or $k=1$).
By \ref{thm:typeisrank}, the rank of the vectors $\{1,e_1,\dots,e_{i-1}\}$ is $i$, and so $m$ is of length $(n,k)$.
\end{proof}

\section{Bimodules}

Right $\mathbb{O}$-modules can be defined similarly to left $\mathbb{O}$-modules.
An $\mathbb{O}$-bimodule should be a left and right $\mathbb{O}$-module with some type of compatibility between the two structures.
In \cite{schafer1952}, a \textbf{representation} for an alternative algebra is introduced that is equivalent to the following definition.

\begin{definition}
An $\mathbb{O}$-bimodule is a left and right module over $\mathbb{O}$ where
\[
[p,q,m]=[q,m,p]=[m,p,q]\quad\forall p,q\in\mathbb{O},m\in M.
\]
We shall denote by $\assoc_L(M)$ (and respectively, $\conjassocL(M)$) the associative part (respectively, conjugate associative part) of $M$ as a left module, and $\assoc_R(M)$ (respectively, $\conjassocR(M)$) will represent the same parts of $M$ as a right module.
\end{definition}

The left $\mathbb{O}$-module $\mathbb{O}$ can be endowed with a natural bimodule structure, however, as it will be clear later on, $\overline{\mathbb{O}}$ cannot be a bimodule.
As it turns out, the bimodule condition is strong enough that a bimodule cannot have a conjugate associatve part.
Furthermore, if it does not have a conjugate associative part, it can be endowed with a unique, natural bimodule structure as the direct sum of bimodules isomorphic to $\mathbb{O}$.
We will first check the left and right associative parts of such a module.

\begin{lemma}
\label{thm:equassoc}
Assume $M$ is a bimodule.
Then $\assoc_L(M)=\assoc_R(M)$.
Furthermore, if $m$ is in the associative part, $(pm)q=p(mq)$ for any $p,q\in\mathbb{O}$.
\end{lemma}

\begin{proof}
We have $m\in\assoc(M)$ if and only if $[a,b,m]=0$ and $m\in\assoc_R(M)$ if and only if $[m,a,b]=0$ for all $a,b\in\mathbb{O}$.
These two are equivalent.
For such an $m$, $[p,m,q]=0$ as well, which means that $(pm)q=p(mq)$.
%
\end{proof}

Henceforth $\assoc(M)$ will denote the associative part of a bimodule independently of whether we consider it as a left or right module.

We want to show that if $m\in\assoc(M)$, we have $pm=mp$.
For this, we need the following lemma.

\begin{lemma}
Let $m\in\assoc(M)$ and $u$, $v$, $w\in\mathbb{O}$.
Then $[u,v,mw]=[mw,u,v]=m[w,u,v]=m[u,v,w]$.
Furthermore, $u(vmw)=(uv)mw-m[u,v,w]$.
\end{lemma}

\begin{proof}
The first equality is a consequence of the definition of a bimodule and the last equality is a property of $\mathbb{O}$, so we only need to check $[mw,u,v]=m[w,u,v]$.
By definition the left hand side is equal to $((mw)u)v-(mw)(uv)$.
Since $m\in\assoc(M)$, we have $(mx)y=m(xy)$ for arbitrary $x$, $y\in\mathbb{O}$.
Hence the associator can be rewritten as $m((wu)v)-m(w(uv))$, which proves the first statement.
The second statement can be proven by expanding $[u,v,mw]$.
\end{proof}

\begin{lemma}
If $m$ is in the associative part of the bimodule $M$, $pm=mp$ for any $p,q\in\mathbb{O}$.
\end{lemma}

\begin{proof}
We may assume that $p$ is purely imaginary by checking separately for the real and imaginary parts.
Let us choose $q$ and $r\in\mathbb{O}$ such that $[p,q,r]\ne 0$ by making sure that $p$, $q$ and $r$ don't appear in a common associative subfield of $\mathbb{O}$.

Since $M$ is a left module, we have $[p,p,qmr]=0$, which we will use to derive the statement.
We have $[p,p,qmr]=(p^2)(qmr)-p(p(qmr))$ by definition.
Using the previous lemma repeatedly, we can rewrite this as
$(p^2q)mr-m[p^2,q,r]-(p(pq))mr+m[p,pq,r]+pm[p,q,r]$.

First notice that the term $(p(pq))mr=(p^2q)mr$, so it cancels out with the first term.
Also since $p$ is pure imaginary, its square is real, and so $[p^2,q,r]=0$.
Hence we have $m[p,pq,r]+pm[p,q,r]=0$.

We claim that since $p\in\mathbb{O}$ is pure imaginary, we have $[p,pq,r]=-p[p,q,r]$.
The left hand side is $(p^2q)r-p((pq)r)$, while the right hand side is $-p((pq)r)+p^2(qr)$.
Rearranging the equation, we need to show $(p^2q)r=p^2(qr)$.
Since $p$ is pure imaginary, $p^2$ is real, therefore the claim is verified.

Returning to the main proof, since $pm[p,q,r]=-m[p,pq,r]$ holds, by the previous claim we also have $pm[p,q,r]=-m[p,pq,r]=mp[p,q,r]$.
Since $q$ and $r$ were chosen so that $[p,q,r]\ne 0$, we can cancel out by $[p,q,r]$ on both sides of the equation, leaving us with $pm=mp$.
\end{proof}

Now we shall turn our attention to the conjugate associative part of a bimodule.

\begin{lemma}
\label{thm:conjassocR}
Assume that $M$ is a bimodule.
Let $m\in\conjassocR(M)$.
Then $m\in\assoc_L(M)+\conjassocL(M)$.
\end{lemma}

\begin{proof}
Choose two distinct imaginary units from the basis of $\mathbb{O}$, $e_i$ and $e_j$.
We have $[e_i,e_j,m]=[m,e_i,e_j]$.
Since $m\in\conjassocR(M)$, $[m,e_i,e_j]=(me_i)e_j-m(e_ie_j)=m(e_je_i-e_ie_j)$.
Since $e_i$ and $e_j$ are pure and distinct, their product is also an imaginary unit $e_k=e_je_i$ (or $e_k=-e_je_i$ in the octonionic basis, but by replacing $e_k$ by $-e_k$, it reduces to this case).
This gives $2me_k$ for the right hand side, and we can substitute $e_j=-e_ke_i$.

On the other hand, $[e_i,e_j,m]=(e_ie_j)m-e_i(e_jm)=-e_km+e_i((e_ke_i)m)$.
Since the right hand side does not depend on $e_i$, the left hand side must not depend on it either, as long as $e_i$ is an imaginary unit distinct from $e_k$.
We will only consider the term that contains $e_i$, $e_i((e_ke_i)m)$.

We can write all elements of the left $\mathbb{O}$-module $M$ as $\mathbb{O}$-linear combinations of left associative and conjugate associative elements.
Let $m=\sum_t p_tm_t$ for $p_t\in\mathbb{O}$ and $m_t\in\assoc_L(M)\cup\conjassocL(M)$ $\mathbb{R}$-linearly independent elements.
Then $e_i((e_ke_i)m)=\sum_t e_i((e_ke_i)(p_tm_t))$.
We will show that all the coefficients $p_t$ are in fact real.

Using the fact that all $m_t$ are either associative or conjugate associative, we can rewrite the sum as
\[
\sum_t e_i((e_ke_i)(p_tm_t))=\sum_{m_t\in\assoc_L(M)} (e_i((e_ke_i)p_t))m_t+\sum_{m_t\in\conjassocL(M)} ((p_t(e_ke_i))e_i)m_t.
\]
By \ref{thm:decomp}, this is the unique way to write this element as a linear combination of $m_t$.
Therefore if it is independent from the choice of $e_i$, all coefficients must be so too, and it suffices to look at the terms separately.

Take a vector $m_t$ and first assume that $m_t\in\conjassocL(M)$.
Consider $(p_t(e_ke_i))e_i$.
It can be checked that for any three octonionic units $e_a$, $e_b$, $e_c$, we have $(e_ae_b)e_c=\pm e_a(e_be_c)$.
If the sign is positive, the three elements belong to a common associative subfield.
In fact, given two units $e_i$ and $e_k$, any element $p_t\in\mathbb{O}$ can be decomposed into two parts $p_t'+p_t''$ where $(p_t'e_k)e_i=p_t'(e_ke_i)$ and $(p_t''e_k)e_i=-p_t''(e_ke_i)$.
Then the term becomes $(p_t''-p_t')e_k$.
Unless $p_t$ belongs to the associative subfield generated by $e_i$ and $e_k$, this decomposition will depend on $e_i$.
Therefore we must have $p_t$ must be part of all such subfields for all values of $e_i$, but also for all values of $e_k$, since $p_t$ only depends on $m$, not $e_k$.
Since $e_k$ and $e_i$ were presumed arbitrary distinct imaginary units from the basis, this means that $p_t$ must be part of all quaternionic subfields of $\mathbb{O}$, which means that all $p_t\in\mathbb{R}$.

With a similar argument for $m_t\in\assoc_L(M)$, we get $\sum_t p_tm_t=m\in\assoc_L(M)+\conjassocL(M)$.
\end{proof}

\begin{lemma}
For a given bimodule $M$, $\conjassocL(M)=\conjassocR(M)$.
\end{lemma}

\begin{proof}
We will show that $\conjassocR(M)\subseteq\conjassocL(M)$.
Choose an element $m\in\conjassocR(M)$.
By \ref{thm:conjassocR} we know that $m\in\assoc_L(M)+\conjassoc_L(M)$.
Let us write the decomposition $m=\alpha+\beta$ where $\alpha\in\assoc_L(M)$ and $\beta\in\conjassoc_L(M)$.
By \ref{thm:equassoc}, we have $\alpha\in\assoc_R(M)$ as well.
In the following paragraphs, $i$, $j$, $k$ will denote the basis of a subfield of $\mathbb{O}$ isomorphic to the quaternions.

Since $m\in\conjassocR(M)$, we have $(mi)j=m(ji)=-mk$.
But also $(mi)j=((\alpha+\beta)i)j=\alpha(ij)+(\beta i)j$.
Therefore $(\beta i)j=-mk-\alpha k=-2\alpha k-\beta k$.

Consider $[i,j,\beta]=[\beta,i,j]$.
The left hand side is $(ij)\beta-i(j\beta)=k\beta-(ji)\beta=2k\beta$.
The right hand side is $(\beta i)j-\beta(ij)$ which by the previous paragraph is $-2\alpha k-2\beta k$.
Therefore $\beta k=-k(\alpha+\beta)$.

Now take an arbitrary pure imaginary unit $v$ and consider $[k,v,\beta]=[v,\beta,k]$.
The left hand side is $(kv)\beta-k(v\beta)=(kv-vk)\beta$ and the right hand side is $(v\beta)k-v(\beta k)$.
By the previous paragraph, this is $(v\beta)k+v(k\beta)+v(k\alpha)$.
Therefore $(v\beta)k=(kv-vk)\beta-(kv)\beta-(vk)\alpha=-(vk)(\alpha+\beta)$.

Finally, consider $[k,i,j\beta]=[i,j\beta,k]$.
The left hand side is $(ki)(j\beta)-k(i(j\beta))=(j(ki)-(ji)k)\beta=-2\beta$.
The right hand side is $(i(j\beta))k-i((j\beta)k)$.
Applying the result of the previous paragraph to $v=j$, we get $(j\beta)k=-i(\alpha+\beta)$.
Applying the result of the previous paragraph to $v=k$, we get $(i(j\beta))k=-k\beta k=-(\alpha+\beta)$.
Thus the right hand side simplifies to $-2\alpha-2\beta$.
Therefore $\alpha=0$, and $m\in\conjassocL(M)$.
\end{proof}

From now on we will denote $\conjassocL(M)=\conjassocR(M)$ by $\conjassoc(M)$ for bimodules $M$.

\begin{lemma}
\label{thm:conjclosed}
Let $m\in\conjassoc(M)$.
Then $(am)b=(a\overline{b})m$.
\end{lemma}

\begin{proof}
First we will prove it for $a=1$.
We may decompose $b$ into a real and a purely imaginary part and check it by parts.
Since it is trivial for real numbers, we may assume $b$ is purely imaginary.
Consider $[u,v,m]=[m,u,v]$ for arbitrary $u$ and $v$.
The left hand side is $(uv)m-u(vm)=(uv-vu)m$, while the right hand side is $(mu)v-m(uv)=m(vu-uv)$.
By choosing $u$ and $v$ in a way that $b=vu-uv$, we get $\overline{b}m=-bm=mb$.

Now consider the general case of $[b,a,m]=[a,m,b]$.
Once again we will only check for purely imaginary $b$.
The right hand side is $[a,m,b]=(am)b-a(mb)$ and $a(mb)=a(\overline{b}m)=(\overline{b}a)m=-(ba)m$, giving $(am)b+(ba)m$ for the right hand side.
The left hand side is $(ba)m-b(am)=(ba-ab)m$.
Reordering the terms, we get
$(am)b=(ba-ab)m-(ba)m=-(ab)m=(a\overline{b})m$.
\end{proof}

\begin{theorem}
Assume $M$ is of type $(\mu,\nu)$ and is also a bimodule.
Then $\nu=0$.
\end{theorem}

\begin{proof}
Let us choose $3$ pure imaginary octonionic units that do not lie in an associative subfield: $u$, $v$, $w$ (for example, using the notations in section 2, $e_1$, $e_2$ and $e_3$) and an element $m\in\conjassoc(M)$.
We will show that $[u,v,wm]=[v,wm,u]$ leads to $m=0$.

The left hand side is $(uv)(wm)-u(v(wm))=(w(uv)-(wv)u)m$.
Since $u$, $v$, $w$ do not lie in an associative subfield, $(wv)u=-w(vu)=w(uv)$, thus $[u,v,wm]=0$.

Now consider the right hand side $(v(wm))u-v((wm)u)=((wv)m)u-v((wm)u)$.
By \ref{thm:conjclosed}, $(am)u=(a\overline{u})m=-(au)m$ for a pure imaginary $u$, therefore $((wv)m)u=-((wv)u)m$ and $((wm)u)=-(wu)m$.
The right hand side thus becomes $-((wv)u)m+v((wu)m)=((wu)v-(wv)u)m$.
Once again, $u$, $v$, $w$ do not lie in an associative subfield, therefore $(wv)u=-w(vu)=w(uv)$ and $(wu)v=-w(uv)$, thus we get $-2(w(uv))m$, which must be zero.
Since the octonions are a division ring, $w(uv)\ne 0$, therefore $m=0$.
\end{proof}

\begin{corollary}
For a given left $\mathbb{O}$-module $M$ of type $(\mu,\nu)$ for arbitrary cardinalities $\mu$ and $\nu$, $M$ admits a bimodule structure if and only if $\nu=0$, or equivalently, $\conjassoc(M)=0$.
For such modules, the bimodule structure is unique and isomorphic to the direct sum of bimodules $\bigoplus_{i\in S}\mathbb{O}$ for $S$ a set of cardinality $\mu$.
\end{corollary}

\section*{Acknowledgements}
The author is grateful to Qinghai Huo for pointing out a mistake in private correspondence in a previous version of the article.


\begin{thebibliography}{XXX}

\bibitem{atiyah1963}
M. F. Atiyah, R. Bott, A. Shapiro,
Clifford Modules,
\textit{Topology}, \textbf{3}, Supplement 1 (1964), 3--38.

\bibitem{huo2019}
Qinghai Huo, Yong Li, and Guangbin Ren,
Classification of left octonion modules,
\emph{Advances in Applied Clifford Algebras},
\textbf{31} (2021), \#11.

\bibitem{schafer1952}
R. D. Schafer,
Representations of Alternative Algebras,
\textit{Transactions of the American Mathematical Society}, \textbf{72}, \textbf{1} (1952), 1--17.

\end{thebibliography}
\end{document}